\documentclass{amsart}
\usepackage{amsthm,amsmath,amssymb, color, graphics}
\usepackage{tikz-cd}
\usepackage{mathrsfs}
\usepackage{adjustbox}
\usepackage{hyperref}
 \usepackage{amsaddr}

\newtheorem{theorem}{Theorem}[section]
\newtheorem{lemma}[theorem]{Lemma}
\newtheorem{proposition}[theorem]{Proposition}
\newtheorem{claim}[theorem]{Claim}
\newtheorem{subclaim}[theorem]{Subclaim}
\newtheorem{corollary}[theorem]{Corollary}
\newtheorem{question}[theorem]{Question}

\newtheorem{fact}{Fact}

\theoremstyle{definition}
\newtheorem{definition}[theorem]{Definition}

\newcommand\ZZ{\mathbb{Z}}

\begin{document}
\title{Nonvanishing higher derived limits without $w\diamondsuit_{\omega_1}$}
\author{Nathaniel Bannister}
\address{Carnegie Mellon University}
\begin{abstract}
    We prove a common refinement of theorems of Bergfalk and of Casarosa and Lambie-Hanson, showing that under certain hypotheses, the higher derived limits of a certain inverse system of abelian groups $\mathbf{A}$ do not vanish. 
    The refined theorem has a number of interesting corollaries, including the nonvanishing of the second derived limit of $\mathbf{A}$ in many of the common models of set theory of the reals and in the Mitchell model. 
    In particular, we disprove a conjecture of Bergfalk, Hru\v{s}\'ak, and Lambie-Hanson that higher derived limits of $\mathbf{A}$ vanish in the Miller model. 
\end{abstract}
\maketitle

\section{Introduction}
The past decade has seen considerable development in the set-theoretic analysis of the higher derived functors of the inverse limit functor. 
Derived limits seek to measure and correct the failures of the inverse limit functor to preserve exact sequences. 
Originating from the theory of \emph{strong homology}, the derived functors of the inverse system $\mathbf{A}$ and its relatives are of particular interest. 
$\mathbf{A}$ is indexed over $\omega^\omega$ (the set of all functions from the natural numbers to the natural numbers) with respect to the ordering of coordinatewise domination; for each index $x$ the group $\mathbf{A}_x$ is defined by
\[\mathbf{A}_x=\bigoplus_{i<\omega}\ZZ^{x(i)}\]
and if $x\leq y$ the map from $\mathbf{A}_y$ to $\mathbf{A}_x$ is the canonical projection map.
We often view $\mathbf{A}_x$ as the group of finitely supported functions from the area under the graph of $x$ into $\mathbb{Z}$. 
$\mathbf{A}$ was first defined and used by Marde\v{s}i\'c and Prasolov in \cite{SHINA} where they showed that if strong homology is \emph{additive} in the sense that it maps disjoint unions to direct sums, then $\lim^n\mathbf{A}=0$ for all $0<n<\omega$. 
More recently, Clausen and Sholze have shown that the vanishing of derived limits of $\mathbf{A}$ has consequences for the vanishing of $\operatorname{Ext}$ groups in their theory of condensed mathematics (see, for instance, \cite[Lecture 4]{Stacks}). 

Many of the early investigations in the late 80s and early 90s into the derived limits of $\mathbf{A}$ focused solely on the \emph{first} derived limit $\lim^1\mathbf{A}$, beginning with \cite{SHINA} where Marde\v{s}i\'c and Prasolov show that the continuum hypothesis implies $\lim^1\mathbf{A}\neq0$. 
In \cite{SHDLST}, Bergfalk begins the investigation into the consistent (non) vanishing of $\lim^n\mathbf{A}$ for $n\geq 2$ by showing that the Proper Forcing Axiom implies that $\lim^2\mathbf{A}\neq0$. 
More precisely, he shows that if $\mathfrak{b}=\mathfrak{d}=\aleph_2$ and $\diamondsuit(S^2_1)$ then $\lim^2\mathbf{A}\neq0$.
In the decade since, a number of results on the vanishing and nonvanishing of the higher derived limits of $\mathbf{A}$ have appeared, including:
\begin{itemize}
    \item In \cite{SVHDL}, Bergfalk and Lambie-Hanson show that consistently from a weakly compact cardinal, $\lim^n\mathbf{A}=0$ for all $n>0$.
    \item In \cite{SVHDLWLC}, Bergfalk, Lambie-Hanson, and Hru\v{s}\'ak remove the large cardinal hypothesis of \cite{SVHDL}.
    \item In \cite{NVHDL}, Veli\v{c}kovi\'c and Vignati prove that for each $1\leq n<\omega$, it is consistent that $\lim^n\mathbf{A}\neq0$. 
    More precisely, they show that if $\mathfrak{b}=\mathfrak{d}=\aleph_n$ and $w\diamondsuit(S^{k+1}_k)$ holds for all $k<n$ then $\lim^n\mathbf{A}\neq0$. 
    \item In \cite{NVHDLWS}, Casarosa refines the Veli\v{c}kovi\'c and Vignati result by showing that with the same weak diamond hypotheses, a weakening of $\mathfrak{b}=\mathfrak{d}$ suffices. 
    \item In \cite{SNVHDL}, Casarosa and Lambie-Hanson remove the hypothesis of $\mathfrak{b}=\mathfrak{d}$ entirely to show that if $\mathfrak{d}=\aleph_2$ and $w\diamondsuit(S^{k+1}_k)$ holds for all $k<n$ then $\lim^n\mathbf{A}\neq0$. 
    They then obtain nonvanishing of derived limits from square sequences to prove that consistently, $\lim^n\mathbf{A}\neq$ for all $2\leq n<\omega$. 
\end{itemize}
The goal of this paper is to continue this line of research. 
The motivating question is the following:
\begin{question}
    What is the least value of $2^{\aleph_0}$ compatible with the assertion that $\lim^n\mathbf{A}=0$ for all $n>0$?
\end{question}
This question remains open. The smallest known upper bound is $\aleph_{\omega+1}$, obtained by Bergfalk, Hru\v{s}\'ak, and Lambie-Hanson in \cite{SVHDLWLC} and the largest known lower bound is $\aleph_2$ as the Marde\v{s}i\'c and Prasolov show in \cite{SHINA} that the continuum hypothesis implies $\lim^1\mathbf{A}\neq0$. 
Dow, Simon, and Vaughan show in \cite{SHPFA} that actually a weakening of $\mathfrak{d}=\aleph_1$ suffices. 

We show that the second derived limit of $\mathbf{A}$ does not vanish in several proposed candidate models, including the Miller model (see \cite[Question 7.4]{SVHDLWLC}) and the Mitchell model (see \cite[Question 7.7]{ADLCM})
To this end, we prove a common generalization of \cite[Theorem 4.1]{SHDLST} and \cite[Theorem A(2)]{SNVHDL}; note the ``in particular'' conclusion holds since $\diamondsuit(S^2_1)$ implies $\mathfrak{d}\leq 2^{\aleph_0}\leq\aleph_2$:
\begin{theorem} \label{intro_main}
    If $\mathfrak{d}=\aleph_n$ and $w\diamondsuit(S^{k+1}_k)$ holds for $1\leq k<n$ then $\lim^n\mathbf{A}\neq0$. 
    In particular, $\diamondsuit(S^2_1)$ implies either $\lim^1\mathbf{A}\neq0$ or $\lim^2\mathbf{A}\neq0$. 
\end{theorem}
We then complete the analysis of derived limits in the Mitchell model to show that the nonvanishing derived limits in the Mitchell model coincide with the ones which do not vanish under the Proper Forcing Axiom. 
\begin{theorem} \label{mit_thm_intro}
    In the Mitchell model, $\lim^n\mathbf{A}\neq0$ if and only if $n=0,2$
\end{theorem}
An important part of our proof of Theorem \ref{mit_thm_intro} is that Cohen forcing does not trivialize any nontrivial $1$-coherent families. 
We show that several other forcings similarly do not trivialize $1$-coherent families.  

The structure of the paper is as follows: In Section \ref{prelim_section}, we review the preliminary material that we will use repeatedly throughout the paper. 
In Section \ref{diamond_section}, we prove a weak version of Theorem \ref{intro_main} namely:
\begin{theorem} \label{diamond_thm_intro}
    If $\mathfrak{d}=\aleph_2$ and $\diamondsuit(S^2_1)$ holds then $\lim^2\mathbf{A}\neq0$. 
\end{theorem}
In Section \ref{certain_models_section}, we analyze a number of models that realize the hypotheses of Theorem \ref{diamond_thm_intro}, including the Miller and Mitchell models. 
In Section \ref{wd_section}, we give a full proof of Theorem \ref{intro_main} and outline a new proof of \cite[Theorem A(1)]{SNVHDL}:
\begin{theorem}[Casarosa, Lambie-Hanson {\cite[Theorem A(1)]{SNVHDL}}]
    Suppose $\mathfrak{d}=\aleph_n$ and $\langle G_i\mid i<\omega_n\rangle$ are nonzero abelian groups. 
    There is a nontrivial $\bigoplus_{i<\omega_n}G_i$-valued $n$-coherent family. 
\end{theorem}
In Section \ref{Mitchell_section}, we complete the analysis of $\lim^n\mathbf{A}$ in the Mitchell model and prove Theorem \ref{mit_thm_intro}. C. Lambie-Hanson has improved the large cardinal needed for this argument from a weakly compact cardinal to an inaccessible cardinal and we include his argument with his kind permission. 
In Section \ref{pres_section}, we present a number of forcings which do not trivialize any nontrivial $1$-coherent families. 
Finally, we present several related open questions. 

\section{Preliminaries} \label{prelim_section}
For $x\in\omega^\omega$, we write $I_x$ to mean $\{(n,m)\in\omega\times\omega\mid m\leq x(n)\}$ and for $\vec x\in (\omega^\omega)^{<\omega}$, we write $I_{\vec x}$ to mean $I_{\bigwedge \vec x}$ where $\bigwedge\vec x$ is the pointwise meet of $\vec x$. 
If $f,g$ are functions each with a domain a subset of $\omega\times\omega$, we write $f=^mg$ to mean that whenever $n>m$ and $i<\omega$ satisfies $(n,i)\in\operatorname{dom}(f)\cap\operatorname{dom}(g)$, we have $f(n,i)=g(n,i)$. 
We write $f=^*g$ to mean that $\{x\in\operatorname{dom}(f)\cap\operatorname{dom}(g)\mid f(x)\neq g(x)\}$ is finite. 
Oftentimes, we will pointwise add finitely many functions with different domains. 
When this occurs, we implicitly mean to restrict each function to the intersection of all of their domains. 
If $\vec x$ is a finite sequence and $i<|\vec x|$, we write $\vec x^i$ to mean the subsequence of $\vec x$ with the $i$th entry removed. 
\begin{definition} \label{coherent_def}
    If $H$ is an abelian group, $1\leq n<\omega$, and $X\subseteq \omega^\omega$, an \emph{$H$ valued $n$-family indexed by $X$} is a family $\Phi$ of functions $\Phi_{\vec x}\colon I_{\vec x}\to H$ for each $\vec x\in X^n$. 
    If $\Phi$ is an $n$-family indexed by $X$, we write $d\Phi$ to be the $(n+1)$-family indexed by $X$ with \[(d\Phi)_{\vec x}=\sum_{i=0}^n(-1)^i\Phi_{\vec x^i};\]
        Note here that we follow the convention that all functions are restricted to the common intersection of their domains.
    An $n$-family $\Phi$ indexed by $X$ is
    \begin{itemize}
        \item \emph{alternating} if for each $\vec x\in X^n$ and permutation $\sigma\in S_n$, $\Phi_{\vec x}=\operatorname{sgn}(\sigma)\Phi_{\sigma(\vec x)}$; 
        \item \emph{$n$-coherent} if it is alternating and for each $\vec x\in X^{n+1}$, $d\Phi_{\vec x}=^*0$; 
        \item \emph{$n$-trivial} if 
        \begin{itemize}
            \item 
        $n=1$ and there is a $\Psi\colon\omega\times\omega\to H$ such that for all $x\in X$, $\Psi=^*\Phi_x$ or 
        \item $n>1$ and there is an alternating $(n-1)$-family $\Psi$ indexed by $X$ such that for all $\vec x\in X^n$, $\Phi_{\vec x}=^*d\Psi_{\vec x}.$
        \end{itemize}
    \end{itemize}
    If $H$ is unspecified, we mean $H=\mathbb{Z}$ and if $X$ is unspecified we mean $X=\omega^\omega$. 
\end{definition}

Note that every $n$-trivial family is $n$-coherent. Whether the converse is true is independent of ZFC and is the subject of ongoing research. 

Coherent and trivial families are closely related to the derived limits of the inverse system of abelian groups $\mathbf{A}$. 
We will implicitly use the following fact relating the derived limits of $\mathbf{A}$ and coherent families; see, for instance, \cite[Theorem 3.3]{SHDLST} for a proof. 

\begin{fact}
    For $n>0$, $\lim^n\mathbf{A}=0$ if and only if every $n$-coherent family is $n$-trivial. 
\end{fact}

We need the following variation of Goblot's vanishing theorem. See, for instance, \cite[Proposition 2.7]{SNVHDL} for a proof.
\begin{theorem} \label{gob_vanish}
    Suppose that $H$ is an abelian group, $0<n<\omega$, and $|X|<\aleph_n$. 
    Any $H$ valued $n$-coherent family indexed by $X$ is trivial. 
\end{theorem}
The following is often useful when we define coherent families recursively.
\begin{proposition} \label{triv_ext_prop}
    Suppose $\Phi$ is an $n$-trivial family indexed by $X$. 
    Then $\Phi$ extends to an $n$-trivial family indexed by $\omega^\omega$. 
\end{proposition}
\begin{proof}
    If $n=1$, let $\Psi$ be a trivialization of $\Phi$ and set
    \[\Phi^*_x=\left\{\begin{array}{cc}
         \Phi_x&x\in X  \\
         \Psi\upharpoonright I_x&\text{otherwise} 
    \end{array}\right..\]
    Then $\Phi^*$ is $1$-trivial since $\Psi$ is a trivialization. 
    
    If $n>1$, let $\Psi$ be a trivialization of $\Phi$. 
    Extend $\Psi$ arbitrarily to an alternating family indexed by $\omega^\omega$, and for each $\vec x\in(\omega^\omega)^n\setminus X^n$ let
    \[\Phi_{\vec x}^*=d\Psi_{\vec x}.\]
    Then $\Phi^*$ is $n$-trivial since $\Psi$ is a trivialization.  
\end{proof}

\section{Nonvanishing from $\diamondsuit(S^2_1)$} \label{diamond_section}
The main result of this section is the following refinement of \cite[Theorem 4.1]{SHDLST} by way of removing the hypothesis $\mathfrak{b}=\aleph_2$. 
In Section \ref{wd_section}, we will further strengthen the theorem to allow $w\diamondsuit(S^2_1)$ rather than $\diamondsuit(S^2_1)$ and thereby also refine \cite[Theorem A(2)]{SNVHDL}. 
\begin{theorem} \label{diamond_thm}
$\diamondsuit(S^2_1)+\mathfrak{d}=\aleph_2$ implies that $\lim^2\mathbf{A}\neq0$. 
\end{theorem}
We first recall the definition of $\diamondsuit(S)$ for stationary $S\subseteq\kappa$:
\begin{definition}
    For $\kappa$ regular and $S\subseteq\kappa$ stationary, $\diamondsuit(S)$ asserts that there are $\langle A_\alpha\mid\alpha<\kappa\rangle$ such that $A_\alpha\subseteq\alpha$ for each $\alpha$ and for every $A\subseteq \kappa$, $\{\alpha\in S\mid A\,\cap \,\alpha=A_\alpha\}$ is stationary. 
\end{definition}
If $n<m<\omega$, we write write $S^m_n$ to mean the collection of ordinals less than $\aleph_m$ of cofinality $\aleph_n$. 
We will need the following weakening of triviality:
\begin{definition} \label{coherent_below_g_def}
    If $g\in\omega^\omega$ and $X\subseteq\omega^\omega$, a 
    $1$-coherent family $\Phi$ indexed by $X$ is \emph{trivial below $g$} if there is a $\Psi\colon I_g\to\mathbb{Z}$ such that $\Psi=^*\Phi_x$ for all $x\in X$. 
\end{definition}
Note that a $1$-coherent family $\Phi$ indexed by $X$ is trivial below $g$ if and only if $\Phi$ extends to a coherent family indexed by $A\cup\{g\}$; in fact, a trivialization is an extension of $\Phi$ to $A\cup\{g\}$. 

The following lemma is a consequence of a theorem of Todor\v{c}evi\'c:
\begin{lemma} \label{nontriv_below}
    Suppose $\langle f_\alpha\mid\alpha<\omega_1\rangle$ is a $<^*$-increasing chain and $g\in\omega^\omega$ is not $<^*$-dominated by any $f_\alpha$. 
    There is a $1$-coherent family indexed by $\{f_\alpha\mid\alpha<\omega_1\}$ which is nontrivial below $g$. 
\end{lemma}
\begin{proof}
    We cite \cite[Theorem 3.3.13]{Walks}:
    \begin{theorem}[{\cite[Theorem 3.3.13]{Walks}}] \label{walks_thm}
        For any strictly $\subset^*$-increasing chain $\langle A_\alpha\mid\alpha<\omega_1\rangle$ of subsets of $\omega$, there is a sequence $\langle B_\alpha\mid \alpha<\omega_1\rangle$ of subsets of $\omega$ such that:
        \begin{itemize}
            \item $B_\alpha=^*B_\beta\cap A_\alpha$ whenever $\alpha<\beta<\omega_1$,
            \item There is no $B$ such that $B_\alpha=^*B\cap A_\alpha$ for all $\alpha<\omega_1$. 
        \end{itemize}
    \end{theorem}
    Now, set $A_\alpha=I_{f_\alpha\wedge g}$. 
    The sequence $\langle A_\alpha\mid\alpha<\omega_1\rangle$ is strictly $\subseteq^*$-increasing: for weakly increasing, if $\alpha<\beta$ and $f_\alpha\leq^mf_\beta$ then $A_\alpha\setminus(m\times\omega)\subseteq A_\beta$ and for the strict part, whenever $f_\alpha(k)<f_\beta(k)<g(k)$, $(k,f_\alpha(k)+1)\in A_\beta\setminus A_\alpha$. 
    Note that $A_\alpha\subseteq\omega\times\omega$ rather than $\omega$, though the difference is only cosmetic. 
    Let $\langle B_\alpha\mid\alpha<\omega_1\rangle$ be as in Theorem \ref{walks_thm} and set
    \[\Phi_{f_\alpha}(i,j)=\left\{\begin{array}{cc}
         1&(i,j)\in B_\alpha  \\
         0&\text{otherwise} 
    \end{array}\right..\]
    Then $\Phi$ is coherent by the first conclusion of Theorem \ref{walks_thm} and nontrivial below $g$ by the second. 
\end{proof}
We now prove Theorem \ref{diamond_thm}.
\begin{proof}
    Fix $\langle M_\alpha\mid\alpha<\omega_2\rangle$ increasing, continuous elementary substructures of $H_\theta$ such that
    \begin{itemize}
        \item for each $\alpha<\omega_2$, $|M_\alpha|=\aleph_1$;
        \item $\omega^\omega=\bigcup_{\alpha<\omega_2}M_\alpha\cap \omega^\omega$;
        \item for each $\alpha<\omega_2$, $\langle M_i\mid i\leq \alpha\rangle\in M_{\alpha+1}$ (that is, the $M_\alpha$ are \emph{internally approaching}).   
    \end{itemize}
    We define $\Phi_{xy}$ for $x,y\in M_\alpha\cap\omega^\omega$ by recursion on $\alpha$. 
    An easy coding argument yields that $\diamondsuit(S^2_1)$ implies there are $\langle\Psi^\alpha_x\mid\alpha<\omega_2,x\in M_\alpha\cap\omega^\omega\rangle$ with $\Psi^\alpha_x\colon I_x\to\mathbb{Z}$ such that for any family $\langle\Psi_x\mid x\in\omega^\omega\rangle$ with each $\Psi_x\colon I_x\to\mathbb{Z}$, there are stationarily many $\alpha\in S^2_1$ satisfying that for all $x\in M_\alpha\cap\omega^\omega$, $\Psi_x=\Psi^\alpha_x$. 

    For $x,y\in M_0$, we define $\Phi_{x,y}$ to be the constant $0$ function. 
    Suppose now that $\alpha\not\in S^2_1$ or that $\Psi^\alpha$ does not trivialize $\Phi\upharpoonright M_\alpha$. 
    By Theorem \ref{gob_vanish}, $\Phi\upharpoonright M_\alpha$ is trivial and we extend $\Phi$ to a $2$-coherent family indexed by $M_{\alpha+1}\cap\omega^\omega$ using Proposition \ref{triv_ext_prop}. 

    Now suppose $\alpha\in S^2_1$ and $\Psi^\alpha$ trivializes $\Phi\upharpoonright M_\alpha$. Note that there is a $<^*$-increasing chain of order type $\omega_1$ which is cofinal in $M_\alpha\cap\omega^\omega$ since $\operatorname{cof}(\alpha)=\omega_1$ and if $\beta<\alpha$ then any countable subset of $M_\beta\cap\omega^\omega$ has a $<^*$ upper bound in $M_{\beta+1}$. 
    Since $\mathfrak{d}=\aleph_2$, we may fix some $x\in M_{\alpha+1}\cap\omega^\omega$ which is not $<^*$-dominated by any element of $M_\alpha$. 
    By Lemma \ref{nontriv_below}, there is a $1$-coherent family $\Theta=\langle\theta_{y}\mid y\in M_{\alpha}\cap\omega^\omega\rangle$ which is nontrivial below $x$. 
    For $y,z$ not both in $M_{\alpha}$, let $\Phi_{yz}=\Psi^\alpha_z+\theta_z-\Psi^\alpha_y-\theta_y$ where $\Psi^\alpha,\theta$ are defined as the constant $0$ functions on indices not in $M_\alpha$ - note that extending $\Theta$ in this way \emph{does not} maintain coherence. 
    We now claim that $\Psi^\alpha$ does not extend to a trivialization of $\Phi$ on $M_{\alpha+1}$. Indeed, if $\Psi^*$ were any such trivialization then for any $y\in M_\alpha\cap\omega^\omega$, 
    \[
    \begin{aligned}
      -\Psi^*_x&=-\Psi^*_x+\Psi^\alpha_y-\Psi^\alpha_y\\
      &=^*\Phi_{xy}-\Psi^\alpha_y\\
      &=\Psi^\alpha_y+\theta_y-\Psi^\alpha_y\\
      &=\theta_y
    \end{aligned}
    \]
    where the first equality is adding $0$, the second is using that $\Psi^*$ is a trivialization of $\Phi$ extending $\Psi^\alpha$, the third is the definition of $\Phi_{xy}$, and the last is canceling identical terms with opposite signs. 
    That is, $-\Psi^*_x$ is a trivialization of $\Theta$ below $x$, contradicting the choice of $\Theta$ to be nontrivial below $x$.
    Now, if $\Phi$ were trivial and $\Psi$ were a trivialization, we could find some $\alpha\in S^2_1$ such that $\Psi^\alpha_{y}=\Psi_{y}$ for all $y\in M_\alpha\cap\omega^\omega$. 
    But then $\Psi$ would be a trivialization of $\Phi$ on $M_{\alpha+1}$ extending $\Psi^\alpha$, which we just showed was impossible. 
\end{proof}

\section{Nonvanishing derived limits in certain models} \label{certain_models_section}
We now give a number of consequences of Theorem \ref{diamond_thm}, including nonvanishing of derived limits in all the standard models from set theory of the reals. 
Our first corollary is the following, as noted in the introduction:
\begin{corollary} \label{diamond_dich_cor}
    $\diamondsuit(S^2_1)$ implies either $\lim^1\mathbf{A}\neq0$ or $\lim^2\mathbf{A}\neq0$. 
\end{corollary}
We obtain the following consequence, whose hypotheses are satisfied by any $\aleph_2$-cc forcing of cardinality $\aleph_2$: 
\begin{corollary} \label{forcings}
    Assume $\diamondsuit(S^2_1)$ holds and let $\mathbb{P}$ be any forcing of cardinality $\aleph_2$ that does not collapse $\aleph_2$ and preserves the stationarity of all stationary subsets of $S^2_1$. 
    After forcing with $\mathbb{P}$, either $\lim^1\mathbf{A}\neq0$ or $\lim^2\mathbf{A}\neq0$. 
\end{corollary}
\begin{proof}
    We make use of the following standard preservation fact:
    \begin{fact} \label{lifting}
        Assume $\kappa$ is regular, $S\subseteq\kappa$ is stationary, and $\diamondsuit(S)$ holds. 
        If $\mathbb{P}$ is a forcing of cardinality $\kappa$ which preserves the regularity of $\kappa$ and the stationarity of any stationary $T\subseteq S$, then $\diamondsuit(S)$ still holds after forcing with $\mathbb{P}$.
    \end{fact}
    \begin{proof}
        We may assume that the underlying set of $\mathbb{P}$ is $\kappa$. 
        By $\diamondsuit(S)$, there are $B_\alpha\subseteq\alpha\times\alpha$ for each $\alpha<\kappa$ such that whenever $X\subseteq\kappa\times\kappa$, there are stationarily many $\alpha\in S$ such that $X\cap(\alpha\times\alpha)=B_\alpha$. 
        Now, in $V[G]$, we let $A_\alpha$ be \[\{\beta<\alpha\mid\exists \gamma\in G\cap\alpha\;(\gamma,\beta)\in B_\alpha\}.\]
        Given $X\subseteq\kappa$, let $\dot\tau$ be a name such that $\dot\tau(G)=X$. 
        Let $T\subseteq S$ be \[\{\alpha\in S\mid B_\alpha=\{(\beta,\gamma)\mid\beta,\gamma<\alpha\wedge \beta\Vdash \gamma\in\dot\tau\}\}.\]
        Note that $T$ is stationary in $V$ and therefore in $V[G]$. 
        Let \[C=\{\gamma<\kappa\mid \forall\alpha<\gamma\;\exists \beta\in G\cap\gamma\;(\beta\,||\,\alpha\in\dot\tau)\}.\]
        Then $C$ is a club in $V[G]$ and if $\gamma\in T\cap C$ then $A_\gamma=X\cap \gamma$. 
    \end{proof}
    Now, suppose $\mathbb{P}$ is as in Corollary \ref{forcings}. 
    If $\mathbb{P}$ collapses $\aleph_1$ then by Fact \ref{lifting}, $V^\mathbb{P}$ satisfies $\diamondsuit_{\omega_1}$ and, in particular, $CH$ so that $\lim^1\mathbf{A}\neq0$. 
    Otherwise, $\diamondsuit(S^2_1)$ holds in $V^\mathbb{P}$ by Fact \ref{lifting} and the proof is complete by Corollary \ref{diamond_dich_cor}. 
\end{proof}

For the Mitchell model, we do not even need $\diamondsuit$ in the ground model; in fact, $\diamondsuit^+_{\omega_2}$ holds in the Mitchell model. Recall that a $\diamondsuit^+_\kappa$ sequence is a sequence $\langle S_\alpha\mid \alpha<\kappa\rangle$ such that $|S_\alpha|=|\alpha|$ for each $\alpha$ and for each $A\subseteq\kappa$, there is a club $C$ such that for each $\alpha\in C$, both $A\cap\alpha$ and $C\cap\alpha$ are in $S_\alpha$. 
A standard fact is that $\diamondsuit^+$ implies $\diamondsuit(S)$ for every stationary $S\subseteq \kappa$; see, for instance, \cite[Theorem 7.14]{Kunen}. 
Let $\mathbb{P}_\alpha=\operatorname{Add}(\omega, \alpha)$ and let $\mathbb{F}_\alpha=\operatorname{Add}\left(\omega_1, 1\right)^{V^{\mathrm{P}_\alpha}}$. 
The conditions of Mitchell forcing $\mathbb{M}$ are the pairs $(p, f)$ for which
\begin{itemize}
\item $p \in \operatorname{Add}(\omega, \kappa)$
    \item $f$ is a partial function on $\kappa$ with countable support
    \item $f(\alpha)$ is a $\mathbb{P}_\alpha$-name for a condition in $\mathbb{F}_\alpha$.
\end{itemize}

The ordering on $\mathbb{M}$ is defined by $(q, g) \leq_{\mathbb{M}}(p, f)$ if and only if
\begin{itemize}
    \item $q \leq_{\mathbb{P}} p$,
    \item $\operatorname{dom}(g) \supseteq \operatorname{dom}(f)$, and
    \item for all $\alpha \in \operatorname{supp}(f)$ the restriction $q\left\lceil(\omega \times \alpha)\right.$ forces that $g(\alpha) \leq_{\mathbb{F}_\alpha} f(\alpha)$.
\end{itemize}

\begin{proposition} \label{diamond}
    For $\kappa$ inaccessible, $\diamondsuit^+_{\kappa}$ holds after Mitchell forcing up to $\kappa$. In particular, $\diamondsuit(S)$ holds whenever $S\subseteq\kappa$ is stationary in $V[G]$.
\end{proposition}
\begin{proof}
    For $\alpha\in S$, let $S_\alpha=\mathcal{P}(\alpha)\cap V[G_\alpha]$; we claim the $S_\alpha$ form a $\diamondsuit^+$ sequence. 
    Note that by inaccessibility of $\kappa$, $|S_\alpha|\leq\aleph_1$ in $V[G]$.
    Now, suppose that $x=\dot x(G)$ is a subset of $\kappa$. 
    For each $\alpha<\kappa$, let $A_\alpha$ be a maximal antichain of conditions deciding whether $\alpha\in\dot x$. 
    Since Mitchell forcing has the $\kappa$-cc, $|A_\alpha|<\kappa$. 
    Let 
    \[f(\alpha)=\sup\bigcup_{(p,g)\in A_\alpha}\operatorname{dom}(p)\cup\operatorname{dom}(g);\]
    observe that $f(\alpha)<\kappa$. 
    Let $C_f$ be the club of all ordinals closed under $f$. 
    Then for $\alpha\in C_f$, $x\cap \alpha\in\mathcal{P}(\alpha)\cap V[G_\alpha]$.
    Moreover, since $C_f\in V$, $C\cap \alpha\in\mathcal{P}(\alpha)\cap V[G_\alpha]$.
\end{proof}
We now note that in the Mitchell model, $\mathfrak{d}=\aleph_2$ since every set of reals of cardinality $\aleph_1$ is added at some proper initial stage and then we add a Cohen real. 
From Theorem \ref{diamond_thm} and Proposition \ref{diamond}, we obtain the following corollary:
\begin{corollary} \label{lim2_mitch}
    In the Mitchell model up to an inaccessible cardinal $\kappa$, $\lim^2\mathbf{A}\neq0$. 
\end{corollary}

We now summarize and give a number of models obtained as in Corollary \ref{forcings}. 
\begin{corollary}
    In each of the following models over a suitable ground model (for instance $L$), $\lim^2\mathbf{A}\neq0$:
    \begin{itemize}
        \item The Mitchell model;
        \item The Miller model;
        \item The Laver model;
        \item The Mathias model;
        \item The length $\omega_2$ Cohen model;
        \item The length $\omega_2$ Hechler model.
    \end{itemize}
\end{corollary}

\section{Nonvanishing from weak diamonds} \label{wd_section}
In this section, we prove an upgraded version of Theorem \ref{diamond_thm} that gives a strict generalization of \cite[Theorem A(2)]{SNVHDL} removing the hypothesis of $w\diamondsuit_{\omega_1}$. 
We then sketch how the same method can be used to produce a new proof of \cite[Theorem A(1)]{SNVHDL}. 
We do not have any additional intended applications of this theorem, though the hypotheses are consistent for each $n$ (see, for instance, \cite[Theorem 4.5]{NVHDL}). 
\begin{theorem} \label{weak_diamond_theorem}
$\mathfrak{d}=\aleph_n$ and $\bigwedge_{i=1}^{n-1}w\diamondsuit(S^{i+1}_i)$ implies that $\lim^n\mathbf{A}\neq0$. 
\end{theorem}
We first recall the definition of $w\diamondsuit(S)$. 
\begin{definition}
    Suppose $\kappa$ is a regular cardinal and $S\subseteq \kappa$ is stationary. $w\diamondsuit(S)$ asserts that whenever $F\colon 2^{<\kappa}\to2$, there is a $g\in 2^\kappa$ such that whenever $b\in 2^\kappa$, 
    \[\{\alpha\in S\mid F(b\upharpoonright\alpha)=g(\alpha)\}\]
    is stationary. 
\end{definition}
For instance, if $\langle A_\alpha\mid\alpha<\kappa\rangle$ is a $\diamondsuit(S)$ sequence, then setting $g(\alpha)=F(A_\alpha)$ witnesses $w\diamondsuit(S)$. 

We now extend Definition \ref{coherent_below_g_def} to allow for arbitrary $n$. 
Note that we allow the parameter $g$ to take the value $\omega$; for $g$ the constant function at $\omega$ an $n$-coherent family is nontrivial below $g$ if and only if it is nontrivial as in Definition \ref{coherent_def}. 
\begin{definition} \label{triv_below_def}
If $g\in(\omega+1)^\omega$ and $X\subseteq\omega^\omega$, an $n$-coherent family $\Phi$ indexed by $X$ is \emph{$n$-trivial below $g$} if 
\begin{itemize}
    \item $n=1$ and there is a $\Psi\colon I_g\to\mathbb{Z}$ such that $\Psi=^*\Phi_x$ for all $x\in X$. 
\item  $n>1$ and there is are alternating $\langle\Psi_{\vec y}\mid \vec y\in X^{n-1}\rangle$ such that for all $\vec y\in X^{n-1}$, $\Psi_{\vec y}\colon I_{\bigwedge\vec y\wedge g}\to\mathbb{Z}$, and for all $\vec x\in X^{n}$, $\Phi_{\vec x}=^*d\Psi_{\vec x}$. 
\end{itemize}
\end{definition}

Note that if $g\leq h$ pointwise and $\Phi$ is nontrivial below $g$ then $\Phi$ is nontrivial below $h$. A convenient reformulation of a coherent family being nontrivial below $g$ is the following:
\begin{proposition} \label{nontrivial_below_prop}
    Suppose $X\subseteq\omega^\omega$, $g\in\omega^\omega$, and $\Phi$ is an $n$-coherent family indexed by $X$. 
    $\Phi$ is trivial below $g$ if and only if $\Phi$ extends to an $n$-coherent family indexed by $X\cup\{g\}$. 
    In particular, if $g\in X$ then $\Phi$ is $n$-trivial below $g$. 
\end{proposition}
\begin{proof}
    If $n=1$ then a trivialization is be definition an extension of $\Phi$ to $X\cup\{g\}$ so we assume $n>1$. 

    Suppose that $\Phi$ extends to some $\Phi^*$ indexed by $X\cup\{g\}$. 
    Then $\Psi_{\vec y}=(-1)^{n+1}\Phi^*_{\vec y^\frown g}$ is a trivialization of $\Phi$ below $g$ as for any $\vec x\in X^n$
    \[\sum_{i=0}^{n-1}(-1)^i\Psi_{\vec x^i}=\sum_{i=0}^{n-1}(-1)^i\Phi^*_{(\vec x^i)^\frown g}=\left(\sum_{i=0}^{n}(-1)^i\Phi^*_{(\vec x^\frown g)^i}\right)-(-1)^n\Phi_{\vec x}=^*(-1)^{n+1}\Phi_{\vec x}\]
    where the first equality is the definitions of $\Psi$, the second is adding $0$, and the third is coherence of $\Phi^*$. 

    Now suppose that $\Phi$ is trivial below $g$ as witnessed by $\langle\Psi_{\vec y}\mid\vec y\in X^{n-1}\rangle$. 
    Let $\Psi^*$ to be an arbitrary alternating extension of $\Psi$ to $X\cup\{g\}$ and set 
    \[\Phi^*_{\vec y}=\left\{\begin{array}{cc}
         \Phi_{\vec y}&\vec y\in X^n  \\
         d\Psi^*_{\vec y}&\text{otherwise} 
    \end{array}\right..\].
    Then $\Phi^*$ is readily verified to be an $n$-coherent family indexed by $X\cup\{g\}$ extending $\Phi$. 
\end{proof}
We now prove the following variation of \cite[Theorem 4.1]{SHDLST}. 
\begin{lemma} \label{Jeff}
    Suppose $n\geq1$, $S\subseteq\omega_{n+1}$ is stationary, $w\diamondsuit(S)$ holds, $\langle A_\alpha\mid\alpha<\omega_{n+1}\rangle$ is an increasing continuous chain of subsets of $\omega^\omega$, and $g\colon\omega\to\omega+1$ is such that 
    \begin{itemize}
        \item $A_\alpha$ is closed under finite joins;
        \item $|A_\alpha|=\aleph_n$ for all $\alpha<\omega_{n+1}$;
        \item if $\alpha\in S$, there is an $n$-coherent family $\Phi$ indexed by $A_\alpha$ and $y\in A_{\alpha+1}$
        such that $\Phi$ is nontrivial below $y\wedge g$.
    \end{itemize}
    Then there is a $(n+1)$-coherent family indexed by $\bigcup_\alpha A_\alpha$ which is nontrivial below $g$.
\end{lemma}
\begin{proof}
    By recursion on the length of $s$, we define for each $s\in 2^{<\omega_{n+1}}$ an $(n+1)$-coherent family $\Phi^s$ indexed by $A_{|s|}$, ensuring additionally that 
    \begin{itemize}
        \item if $t\sqsubseteq s$ then $\Phi^s\upharpoonright A_{|t|}=\Phi^t$;
        \item if $t\upharpoonright S=s\upharpoonright S$ and $|t|=|s|$ then $\Phi^t=\Phi^s$.
    \end{itemize}
    
    If $|s|\not\in S$, then since $\Phi^s$ is trivial by Theorem \ref{gob_vanish}, we may extend $\Phi^s$ to an $(n+1)$-coherent family indexed by $A_{|s|+1}$ using Proposition \ref{triv_ext_prop}. 
    We choose the same extension for both $\Phi^{s^\frown 0}$ and $\Phi^{s^\frown 1}$. 
    Now suppose $|s|\in S$. By Theorem \ref{gob_vanish}, we may fix a trivialization $\Psi$ of $\Phi^s$. 
    By hypothesis, we may fix an $n$-coherent family $\Theta$ indexed by $A_{|s|}$ and $y\in A_{|s|+1}$ such that $\Theta$ is nontrivial below $y\wedge g$.
    For $\vec x\in (A_{\alpha+1})^{n+1}\setminus (A_\alpha)^n$, we let
    \[\Phi_{\vec x}^{s^\frown0}=d\Psi_{\vec x}\]
    and
    \[\Phi_{\vec x}^{s^\frown 1}=d[\Psi+\Theta]_{\vec x},\]
    where we define $\Psi$ and $\Theta$ as the constant $0$ functions at indices not all in $A_{|s|}$ (note that this extension of $\Theta$ \emph{is not} coherent). 

    \begin{claim}
        If $\Psi'$ is any trivialization of $\Phi^s$ below $g$, there is at most one $i\in\{0,1\}$ such that $\Psi'$ extends to a trivialization of $\Phi^{s^\frown i}$ below $g$.
    \end{claim}
    \begin{proof}
        Suppose not and let $\Psi^0$ and $\Psi^1$ be extensions of $\Psi'$ which trivialize $\Phi^{s^\frown 0}$ and $\Phi^{s^\frown 1}$ below $g$ respectively. 
        The proof is slightly different depending on whether $n=1$ or $n>1$. 
        If $n=1$, let $\Upsilon\colon I_{y\wedge g}\to \mathbb{Z}$ be $\Psi^0_y-\Psi^1_y$. 
        Then for each $z\in A_{|s|}$,
        \[\begin{aligned}
            \Upsilon&=\Psi^0_y-\Psi^1_y\\
            &=(\Psi^0_y-\Psi'_z)-(\Psi^1_y-\Psi'_z)\\
            &=^*\Phi_{zy}^{s^\frown 0}-\Phi_{zy}^{s^\frown 1}\\
            &=-\Psi_z+(\Psi_z+\Theta_z)\\
            &=\Theta_z,
        \end{aligned}\]
        where the first line is the definition of $\Upsilon$, the second is adding $0$, the third is that $\Psi^0$ and $\Psi^1$ are trivializations of $\Phi^0$ and $\Phi^1$ below $g$ extending $\Psi'$ respectively, and the last is the definition of $\Phi$.
        So $\Upsilon$ trivializes $\Theta$ below $y\wedge g$, contrary to our choice of $\Theta$. 

        Now, suppose $n>1$. For each $\vec z\in (A_{|s|})^{n-1}$, let $\Upsilon_{\vec z}\colon I_{\bigwedge\vec z\wedge y\wedge g}\to\mathbb{Z}$ be $\Psi^0_{\vec z^\frown y}-\Psi^1_{\vec z^\frown y}$. 
        Then for each $\vec x\in (A_{|s|})^n$,
        \[\begin{aligned}
            d\Upsilon_{\vec x}&=
        \sum_{i=0}^{n-1}(-1)^i(\Psi^0_{(\vec x^i)^\frown y}-\Psi^1_{(\vec x^i)^\frown y})\\
        &=(d\Psi^0_{\vec x^\frown y}-(-1)^n\Psi^0_{\vec x})-(d\Psi^1_{\vec x^\frown y}-(-1)^n\Psi^1_{\vec x})\\
        &=^*\Phi^{s^\frown 0}_{\vec x^\frown y}-\Phi^{s^\frown 1}_{\vec x^\frown y}\\
        &=(-1)^n\Psi_{\vec x}-(-1)^n(\Psi_{\vec x}+\Theta_{\vec x})\\
        &=(-1)^{n+1}\Theta_{\vec x},
        \end{aligned}\]
        where the first line is the definition of $d$ and $\Upsilon$, the second is the definition of $d$ applied to $\Psi^0$ and $\Psi^1$, the third is that $\Psi^0$ and $\Psi^1$ extend $\Psi'$ and trivialize $\Phi^{s^\frown 0}$ and $\Phi^{s^\frown 1}$ respectively, the fourth is the definition of $\Phi$, and the last is canceling like terms with opposite signs. 
        So $(-1)^{n+1}\Upsilon$ trivializes $\Theta$ below $y\wedge g$, contrary to our choice of $\Theta$.
    \end{proof}
    
    We claim that for some $h\in2^{\omega_1}$, $\Phi^h=\bigcup_{\alpha<\omega_{n+1}} \Phi^{h\upharpoonright\alpha}$ is nontrivial below $g$. 
    
    For any reasonable choice of encoding attempts at trivialization, if $b\in2^{<\omega_{n+1}}$ codes a trivialization of $\Phi^s$ for some $s$ with $|b|=|s|\in S$, then note that there is a unique such $\Phi^s$. Set $F(b)\in\{0,1\}$ to be such that the trivialization coded by $b$ does not extend to a trivialization of $\Phi^{s^\frown F(b)}$; we define $F(b)$ arbitrarily otherwise. 
    Now by $w\diamondsuit(S)$, we may fix some $h\in2^{\omega_{n+1}}$ such that for all $b\colon\omega_{n+1}\to 2$, there are stationarily many $\alpha\in S$ such that $h(\alpha)=F(b\upharpoonright\alpha)$; it is for this $h$ that we claim $\Phi^h$ is nontrivial. 
    Indeed, if $b$ codes a trivialization for $\Phi^h$, we may fix some $\alpha\in S$ such that $h(\alpha)=F(b\upharpoonright\alpha)$. 
    But then $b\upharpoonright \alpha$ extends to a trivialization of $\Phi^{(h\upharpoonright\alpha)^\frown F(b\upharpoonright\alpha
    )}$, contrary to our choice of $F(b\upharpoonright\alpha)$. 
\end{proof}
We will need the following, which follows by the same proof as \cite[Proposition 4.3]{SNVHDL}. 
The hypotheses are easily achieved with an $M$ of size $\kappa$ for any regular $\kappa<\mathfrak{d}$ by the observation that for any regular $\kappa$, $\operatorname{cof}([\kappa]^{<\kappa},\subseteq)=\kappa$.
We note the hypotheses are automatically satisfied by any $M$ which is the union of an internally approaching chain of regular length $|M|$. 
\begin{lemma} \label{non}
    Suppose $M\prec H_\theta$ and for all $A\in[M]^{<|M|}$ there is $B\in M$ with $A\subseteq B\subseteq M$. 
    Suppose moreover that $g\in\omega^\omega$ is weakly increasing and not $<^*$-dominated by any element of $M$. 
    Then $\operatorname{cof}(\{y\wedge g\mid y\in M\cap\omega^\omega\},<^*)=|M|$.
\end{lemma}
\begin{proof}
    Suppose not and fix $A\subseteq M\cap\omega^\omega$ with $|A|<|M|$ and $\{y\wedge g\mid y\in A\}$ cofinal in $(\{y\wedge g\mid y\in M\cap\omega^\omega\},<^*)$. 
    Then fix $B\in M$ such that $A\subseteq B\subseteq M\cap\omega^\omega$. 
    Then $B$ is not $<^*$-cofinal in $\omega^\omega$ since $g$ is not dominated by any element of $B$ so by elementarity we may fix $z\in M$ which is not dominated by any $y\in B$. 
    By changing $z$ if necessary, we assume that $z$ is increasing. 
    Since $A$ was chosen to be cofinal in $(\{y\wedge g\mid y\in M\cap\omega^\omega\},<^*)$, we now fix some $y\in A$ such that $g\wedge z\leq^* y\wedge z\leq y$. 
    Let $E\in[\omega]^\omega\cap M$ be $\{n\mid y(n)<z(n)\}$; note $E$ is infinite since $y\in A\subseteq B$ and $z$ is not $<^*$-dominated by any element of $B$. Let $n_0<n_1<\ldots$ enumerate $E$. 
    Let $a\in M\cap\omega^\omega$ be defined by $a(n)=z(n_{i+1})$ where $n_i< n\leq n_{i+1}$. 
    Since $g$ is not $<^*$-dominated by any element of $M$, there are infinitely many $i$ such that for some $n$ with $n_i<n\leq n_{i+1}$, $g(n)>a(n)$. 
    For each such $i$ and $n$,
    \[g(n_{i+1})\geq g(n)>a(n)= z(n_{i+1})>y(n_{i+1}).\]
    This contradicts that $g\wedge z\leq^* y$ as $g\wedge z(n_{i+1})=z(n_{i+1})>y(n_{i+1})$ for each such $i$. 
\end{proof} 

The following lemma is the last ingredient we will need. The main case of interest is when $g$ is the constant function $\omega$ so that we simply have a nontrivial $n$-coherent family. 
We need the strengthening for induction purposes. 
\begin{lemma} \label{last}
    Assume $\bigwedge_{1\leq i\leq n-1}w\diamondsuit(S^{i+1}_i)$. Suppose $M\prec H_\theta$ is such that $|M|=\aleph_n$ and for each $A\in [M]^{<\aleph_n}$ there is a $B\in M$ such that $|B|=|A|$ and $A\subseteq B\subseteq M$. 
    Suppose moreover that either 
    \begin{itemize}
        \item $g\in\omega^\omega$ is weakly increasing and not $<^*$-dominated by any element of $M\cap\omega^\omega$ or
        \item $g$ is the constant function $\omega$ and $\operatorname{cof}(M\cap \omega^\omega,<^*)=\aleph_n$.
    \end{itemize}
    Then there is an $n$-coherent family indexed by $M\cap \omega^\omega$ which is nontrivial below $g$.
\end{lemma}
\begin{proof}
    By induction on $n$. 
    When $n=1$, the hypotheses on $M$ yield a cofinal $\omega_1$ sequence in $(M\cap\omega^\omega,<^*)$ so the lemma follows from Lemma \ref{nontriv_below}. 
    
    Suppose now $n>1$. Build an increasing continuous chain $\langle M_\alpha\mid \alpha<\omega_n\rangle$ of elementary substructures of $M$ such that 
    \begin{itemize}
        \item $|M_\alpha|=\aleph_{n-1}$;
        \item whenever $\alpha$ is a successor ordinal and $A\subseteq M_\alpha\cap\omega^\omega$ with $|A|<\aleph_{n-1}$ there is a $B\in M_\alpha$ with $|B|=|A|$ and $A\subseteq B\subseteq M_\alpha$;
        \item $\bigcup_\alpha M_\alpha=M$;
        \item for each $\alpha$ there is a $y_\alpha\in M_{\alpha+1}\cap\omega^\omega$ such that $g\wedge y_\alpha$ is not $<^*$-dominated by any element of $M_\alpha$. 
    \end{itemize}
    Note the last item is possible by Lemma \ref{non} if $g\in\omega^\omega$ or from the hypothesis that $\operatorname{cof}(M\cap \omega^\omega,<^*)=\aleph_n$ if $g$ is the constant function $\omega$.
    We verify the hypotheses of Lemma \ref{Jeff} for $A_\alpha=M_\alpha\cap \omega^\omega$ and $S=S^n_{n-1}$.
    
    The first and second items are immediate. For the third, by replacing $y_\alpha$ if necessary, we may assume that $y_\alpha$ is weakly increasing. The inductive hypothesis yields an $(n-1)$-coherent family indexed by $M_\alpha\cap\omega^\omega$ which is nontrivial below $g\wedge y_\alpha$.  
\end{proof}

We now complete the proof of Theorem \ref{weak_diamond_theorem}:
\begin{proof}
    Fix $M\prec H_\theta$ as in Lemma \ref{last} such that some cofinal subset of $(\omega^\omega,\leq)$ is contained in $M$. 
    Then there is a nontrivial $n$-coherent family indexed by $M\cap\omega^\omega$. 
    The proof is finished since derived limits may be computed along any cofinal suborder. 
\end{proof}

We now note that a similar proof shows the following result (see \cite[Theorem A(1)]{SNVHDL}). 
The motivation is that in both models where we know the derived limits of $\mathbf{A}$ simultaneously vanish, the same holds for the systems $\mathbf{A}[H]$ for any abelian group $H$ (see \cite[Theorem 7.7]{DAHDL} and \cite[Theorem 1.2]{ADLCM}). 
Theorem \ref{nonvanish_d_thm} implies the same cannot hold with a small continuum. 
\begin{theorem} \label{nonvanish_d_thm}
    Suppose $\mathfrak{d}=\aleph_n$ and $\langle G_i\mid i<\omega_n\rangle$ are nonzero abelian groups. 
    There is a nontrivial $\bigoplus_{i<\omega_n}G_i$-valued $n$-coherent family. 
\end{theorem}
The proof is nearly identical, except that we need a replacement for Lemma \ref{Jeff} to avoid weak diamonds. 
The obvious replacement is true, though it requires a new trick using that $\bigoplus_{i<\omega_n}G_i$ as a whole is an incredibly large group yet every element has finite support. 

\begin{lemma} \label{Jeff2}
    Suppose $n\geq1$, $\langle G_i\mid i<\omega_{n+1}\rangle$ are nonzero abelian groups $\langle A_\alpha\mid\alpha<\omega_{n+1}\rangle$ is an increasing continuous chain of subsets of $\omega^\omega$, $S\subseteq\omega_{n+1}$ is stationary, and $g\in(\omega+1)^\omega$ are such that 
    \begin{itemize}
        \item $A_\alpha$ is closed under finite joins;
        \item $|A_\alpha|=\aleph_n$ for all $\alpha<\omega_{n+1}$;
        \item if $\alpha\in S$ and $\langle H_i\mid i<\omega_n\rangle$ are nonzero abelian groups, there is an $\bigoplus_{i<\omega_{n}}H_i$-valued $n$-coherent family $\Phi$ indexed by $A_\alpha$ and $y\in A_{\alpha+1}$ such that $\Phi$ is nontrivial below $y\wedge g$. 
    \end{itemize}
    Then there is an $\bigoplus_{i<\omega_{n+1}}G_i$-valued $(n+1)$-coherent family indexed by $\bigcup_\alpha A_\alpha$ which is nontrivial below $g$.
\end{lemma}
\begin{proof}
    We define $\langle \Phi_{\vec x}\mid\vec x\in (A_\alpha)^{n+1}\rangle$ by recursion on $\alpha$ such that if $\vec x\in A_{\alpha}^{n+1}$ then $\Phi_{\vec x}$ takes values in $\bigoplus_{i<\omega_n\alpha}G_i$. 
    If $\alpha\not\in S$ then since $\Phi\upharpoonright A_\alpha$ is trivial by Theorem \ref{gob_vanish}, we extend $\Phi$ to $A_{\alpha+1}$ using Proposition \ref{triv_ext_prop}. 
    Suppose now that $\alpha\in S$ and let $\Psi^\alpha$ be a trivialization of $\langle\Phi_{\vec x}\mid \vec x\in (A_\alpha)^{n+1}\rangle$ such that each $\Psi^\alpha_{\vec y}$ takes values in $\bigoplus_{i<\omega_n\alpha}G_i$. 
    By hypothesis, let $\Theta^\alpha$ be a $\bigoplus_{\omega_n\alpha\leq i<\omega_{n}(\alpha+1)}G_i$-valued $n$-coherent family indexed by $A_\alpha$ which is nontrivial below $y\wedge g$ for some $y\in A_{\alpha+1}$. 
    For each $\vec x\in (A_{\alpha+1})^{n+1}\setminus (A_\alpha)^{n+1}$, let $\Phi_{\vec x}=d(\Psi^\alpha+\Theta^\alpha)_{\vec x}$, where we extend $\Psi^\alpha$ and $\Theta^\alpha$ to $(A_{\alpha+1})^n\setminus(A_\alpha)^n$ as the constant $0$ functions (note this extension of $\Theta$ is not coherent). 

    Suppose for contradiction that $\Phi$ were trivial below $g$ and let $\Psi$ be a trivialization of $\Phi$ below $g$. 
    By stationarity of $S$, we may fix $\alpha\in S$ such that $\omega_n\alpha=\alpha$ and for all $\vec x\in (A_{\alpha})^{n}$, $\Psi_{\vec x}$ takes values in $\bigoplus_{i<\alpha}G_i$. 
    Then whenever $\vec x\in (A_\alpha)^{n}$, 
    \[(-1)^{n}\Psi_{\vec x}+\sum_{i=0}^{n-1}(-1)^i\Psi_{(\vec x^i)^\frown y}=d\Psi_{\vec x^\frown y}=^*\Phi_{\vec x^\frown y}=(-1)^{n}(\Psi^\alpha_{\vec x}+\Theta^\alpha_{\vec x}).\]
    Since $\Psi_{\vec x}$ and $\Psi^\alpha_{\vec x}$ take values in $\bigoplus_{i<\alpha}G_i$, the restriction of the values of $(-1)^{n}\Psi_{\vec z^\frown y}$ to $\bigoplus_{\alpha\leq i<\alpha+\omega_n}G_i$ trivializes $\Theta^\alpha$ below $y\wedge g$, contradicting our choice of $\Theta^\alpha$.

\end{proof}
\section{Derived limits in the Mitchell model} \label{Mitchell_section}
In this section, we complete the analysis of derived limits in the Mitchell model from an inacessible cardinal to show the following:
\begin{theorem}
    In the Mitchell model, $\lim^n\mathbf{A}\neq0$ if and only if $n=0,2$. 
\end{theorem}
Recall from Corollary \ref{lim2_mitch} that in the Mitchell model, $\lim^2\mathbf{A}\neq0$. 
Moreover, an easy application of Theorem \ref{gob_vanish} is that for $n\geq3$, $\lim^n\mathbf{A}=0$ in the Mitchell model. 
Thus, it remains to show that $\lim^1\mathbf{A}=0$ in the Mitchell model. 
We will fix for the remainder of this section an inaccessible cardinal $\kappa$. For $\alpha<\kappa$, let $\mathbb{P}_\alpha=\operatorname{Add}(\omega, \alpha)$ and let $\mathbb{F}_\alpha=\operatorname{Add}\left(\omega_1, 1\right)^{V^{\mathrm{P} \alpha}}$. 

The bulk of our analysis of the first derived limit in the Mitchell model comes from an analysis of Cohen forcing. 
We first prove the following:
\begin{lemma}
    Suppose $A\subseteq\omega^\omega$ is countably directed under $<^*$ and $\Phi=\langle\Phi_x\mid x\in A\rangle$ is a nontrivial $1$-coherent family. 
    Then $\Phi$ remains nontrivial after adding a Cohen real. 
    Moreover, $\Phi$ does not extend to a coherent family indexed by $A\cup\{c\}$ for $c$ the Cohen real. 
\end{lemma}
\begin{proof}
    Note that the ``moreover'' statement implies the original statement so it suffices to show that $\Phi$ does not extend to $A\cup\{c\}$. 
    Suppose not and fix $p$ forcing that $\Phi_c\colon I_c\to\mathbb{Z}$ satisfies $\Phi_c=^*\Phi_x$ for each $x\in A$. 
    For each $x\in A$, fix $q_x\leq p$ and $n_x<\omega$ such that $q_x\Vdash \Phi_x=^{n_x}\Psi_c$. 
    Since $A$ is countably directed under $<^*$ and Cohen forcing is countable, there is a condition $q$ and an $n<\omega$ such that $B=\{x\in A\mid q_x=q, n_x=n\}$ is cofinal in $(A,<^*)$. 
    We now note that if $x,y\in B$ then $\Phi_x=^{\max(n,|q|)}\Phi_y$. 
    Indeed, for any $m>\max(n,|q|)$ and $i\leq x(m),y(m)$, we may extend $q$ to some $r$ such that $r(m)>x(m),y(m)$. 
    Then
    \[r\Vdash\Phi_x(m,i)=\Phi_c(m,i)=\Phi_y(m,i).\]
    Then
    \[\Psi(m,i)=\left\{\begin{array}{cc}
         \Phi_x(m,i)&m>\max(n,|q|), \text{for some }x\in B, x(m)\geq i  \\
         0&\text{there is no such }x 
    \end{array}\right.\]
    trivializes $\Phi$. 
\end{proof}
The same proof shows that after forcing with $\operatorname{Add}(\omega,\omega)$, a nontrivial $1$-coherent family does not extend to any of the Cohen reals. Usual chain condition arguments then show the following:
\begin{lemma} \label{no_extend}
    Suppose $\lambda$ is any cardinal, $A\subseteq\omega^\omega$ is countably directed under $<^*$ and $\Phi=\langle\Phi_x\mid x\in A\rangle$ is a nontrivial $1$-coherent family. 
    After forcing with $\operatorname{Add}(\omega,\lambda)$, $\Phi$ remains nontrivial. 
    Moreover, $\Phi$ does not extend to a coherent family indexed by $A\cup\{c_\alpha\}$ for any $\alpha$, where $c_\alpha$ is the $\alpha$th Cohen real added. 
\end{lemma}
Using the standard fact that reals added in Mitchell forcing are all added on the Cohens coordinate, Lemma \ref{no_extend} readily implies that Mitchell forcing also does not add trivializations to any nontrivial $1$-coherent family. 
Our next goal is to prove the following Corollary:
\begin{corollary} \label{lim1mitchell}
    In the Mitchell model from an inaccessible $\kappa$, $\lim^1\mathbf{A}=0$.
\end{corollary}
The argument we give is due to C. Lambie-Hanson and refines the large cardinal used from a weakly compact cardinal to an inaccessible cardinal. 
\begin{proof}
    Suppose for simplicity that $\dot \Phi$ is a name for a $1$-coherent family. 
    We show that there is a condition $q$ forcing that $\dot\Phi$ is trivial. 
    First, by inacessibility of $\kappa$, there is a club $C\subseteq \kappa$ such that whenever $\alpha\in C$ has uncountable cofinality, $\dot\Phi$ reflects to an $\mathbb{M}_\alpha$ name $\dot\Phi_\alpha$ for a $1$-coherent family. 
    By Lemma \ref{no_extend}, for each such $\alpha$, $\dot\Phi_\alpha$ must name a trivial $1$-coherent family since it admits an extension to the next Cohen real and, since all reals added by Mitchell forcing are added on the Cohens coordinate, extends to a Cohen generic real in an extension by Cohen forcing. 
    
    For each $\alpha\in C$ of uncountable cofinality, pick a condition $q_\alpha\in\mathbb{M}_\alpha$ and a $\mathbb{P}_\alpha$ name $\dot x_\alpha$ such that $q_\alpha\Vdash\dot x_\alpha$ trivializes $\dot\Phi_\alpha$. 
    Since $\alpha$ has uncountable cofinality, there is a $\beta_\alpha<\alpha$ such that $q_\alpha\in\mathbb{M}_\beta$ and $\dot x_\alpha$ is a $\mathbb{P}_\beta$ name. 
    By Fodor's lemma, the function $\alpha\mapsto\beta_\alpha$ is constant with value $\beta^*$ on a stationary $S\subseteq C\cap\operatorname{cof}(\geq\omega_1)$. 
    By inaccessibility of $\kappa$, $|M_{\beta^*}|<\kappa$ and there are fewer than $\kappa$ many $\mathbb{P}_{\beta^*}$ names for real numbers. 
    In particular, the function $\alpha\mapsto(q_\alpha,\dot x_\alpha)$ is constant on a stationary $T\subseteq S$ with value $(q,\dot x)$. Then $q$ forces that $\dot x$ trivializes $\dot\Phi$, as desired. 
\end{proof}
\section{Preserving nontrivial $1$-coherent families} \label{pres_section}
Crucial to our analysis of the first derived limit in the Mitchell model was that Cohen forcing does not trivialize any nontrivial $1$-coherent families. In this section, we provide a few more examples of forcings that also do not trivialize any nontrivial $1$-coherent families. 
We do not know of any analogous results for $2$-coherent families or, more generally, for $n$-coherent families with $n\geq2$. 

\begin{proposition}
    The following classes of forcings do not add trivializations to nontrivial $1$-coherent families:
    \begin{enumerate}
        \item $\sigma$-distributive forcings
        \item strongly proper forcings
        \item forcings that are a union of strictly fewer than $\mathfrak{b}$ many linked sets
        \item forcings whose square is $\omega^\omega$-bounding
        \item Miller forcing
    \end{enumerate}
\end{proposition}
\begin{proof} We show each item in turn.
    \begin{enumerate}
        \item A trivialization is coded by a real.
        \item Every real added by a strongly proper forcing is (equivalent to) a Cohen real.
        \item Suppose $\Vdash \dot\psi$ is a trivialization for $\Phi$. 
        Let $\langle A_\alpha\mid \alpha<\kappa\rangle$ be a partition of $\mathbb{P}$ by linked sets with $\kappa<\mathfrak{b}$. 
        For each $x\in\omega^\omega$, pick $p_x,k_x$ such that $p_x\Vdash\dot\psi=^{k_x}\Phi_x$ and let $\alpha_x$ be such that $p_x\in A_{\alpha_x}$. 
        We color $\omega^\omega$ by 
        \[x\mapsto(k_x,\alpha_x).\]
        By the $\mathfrak{b}$-directedness of $\leq^*$, there is a $\leq^*$-cofinal $X$ and a $k$ such that 
        \begin{itemize}
            \item 
        $k_x=k$ for each $x\in X$ and 
        \item whenever $x,y\in X$, $\alpha_x=\alpha_y$; in particular, $p_x$ is compatible with $p_y$.
        \end{itemize}
        Then the function
        \[\psi(i,j)=\left\{\begin{array}{cc}
             \Phi_x(i,j)&i\geq k\text{ and there is some }x\in X\text{ with }j\leq x(i)  \\
             0&\text{otherwise} 
        \end{array}\right.\]
        trivializes $\Phi$; note that $\psi$ is well defined since if $x,y\in X$ then $\Phi_x=^k\Phi_y$.
        \item Suppose $\mathbb{P}^2$ is $\omega^\omega$ bounding and $\mathbb{P}$ adds a trivialization for $\Phi$. 
        Let $G\times H$ be generic for $\mathbb{P}^2$ and let $\psi_1$ and $\psi_2$ be trivializations for $\Phi$ added by $G$ and $H$ respectively. 
        Then for some $k$, $\psi_1=^k\psi_2$ as otherwise there is some $x$ such that for infinitely many $i$ and $j\leq x(i)$, $\psi_1(i,j)\neq\psi_2(i,j)$; such an $x$ cannot be bounded by any real in $V$. 
        In particular, $\psi_1\upharpoonright(k,\infty)\times\omega=\psi_2\upharpoonright(k,\infty)\times\omega\in V[G]\cap V[H]=V$ and trivializes $\Phi$.
        \item 
        Recall that Miller forcing consists of trees $T\subseteq\omega^<\omega$ such that every $\sigma\in T$ has an extension which splits infinitely often, ordered by containment. If $p$ is a Miller condition, we write $\operatorname{split}(p)$ to mean the set of splitting nodes in $p$. If $s$ is a node of $p$ we write $p^{[s]}$ to mean all the nodes of $p$ which are either an initial segment of $s$ or which end extend $s$; note that $p^{[s]}$ is a Miller condition extending $p$. 
        Suppose for simplicity that $\Vdash \dot \Psi$ is a trivialization of $\Phi$. 
    The key claim is the following:
    \begin{claim} \label{miller_pres_clm}
        There is a condition $p$, functions $\langle f_s\mid s\in\operatorname{split}(p)\rangle$, and numbers $\langle n^s_i\mid s\in\operatorname{split}(p),s^\frown i\in p\rangle$ such that 
        \begin{itemize}
            \item for each $s,i,$ and $j$, if $i<j$ then $n^s_i<n^s_j$;
            \item for each $s$ and $i$, 
        \[p^{[s^\frown i]}\Vdash\dot\Psi=^{n_i^s}\Phi_{f_s};\]
        \item for each $s$ and each $i<j$, 
        \[p^{[s^\frown j]}\Vdash\dot\Psi\neq^{n_i^s}\Phi_{f_s}.\]
        \end{itemize}
    \end{claim}
    We first note that the claim is sufficient. 
    Indeed, let $f$ be a $<^*$ upper bound for $\{f_s\mid s\in\operatorname{split}(p)\}$ and fix $q\leq p$ and $n$ such that $q\Vdash\dot\Psi=^n\Phi_f$. 
    For $s\in \operatorname{split}(p)$, let $k_s=\max(n,\min\{m\mid f_s\leq^mf\}, \min\{m\mid \Phi_{f_s}=^m\Phi_f\})$. 
    Then whenever $s\in\operatorname{split}(q)$, 
    \[q\Vdash \dot \Psi=^{k_s}\Phi_{f_s}\]
    since $q$ forces that for each $i\geq k$ and $j\leq f_s(k)$,
    \[\Psi(i,j)=\Phi_f(i,j)=\Phi_{f_s}(i,j).\]
    But this means that $q$ can split only finitely often at $s$ as $q$ is incompatible with each $p^{[s^\frown i]}$ satisfying $n^s_i> k_s$, a contradiction. 
    
    The proof of Claim \ref{miller_pres_clm} is, of course, a fusion argument. 
    We argue that for a single condition $q$ and $s\in\operatorname{split}(q)$, there is a condition extending $q$ in which $s$ is still a splitting node but there are $f_s$ and $\langle n_i^s\rangle$ as in Claim \ref{miller_pres_clm}. 
    We use the following general fact which holds for every forcing whatsoever:
    \begin{subclaim}
        For every condition $r$, there is an $f\in{}^\omega\omega$ such that for all $n<\omega$, 
        \[r\not\Vdash\dot\Psi=^n\Phi_f.\]
    \end{subclaim}
    \begin{proof}
        Suppose not and let $r$ be a counterexample. 
        For each $f\in{}^\omega\omega$ fix $n_f$ such that $r\Vdash\dot\Psi=^{n_f}\Phi_f$. 
        Since $(\omega^\omega,<^*)$ is countably directed, there is a $<^*$-cofinal set $A\subseteq\omega^\omega$ and an $n<\omega$ such that $n_f=n$ for each $f\in A$. 
        Then let 
        \[\Psi(i,j)=\left\{\begin{array}{cc}
             \Phi_f(i,j)&i>n \text{ and for some }f\in A\text{, }  f(i)\geq j\\
             0&\text{otherwise} 
        \end{array}\right.\]
        $\Psi$ trivializes $\Phi$ since for each $f\in\omega^\omega$, we may find $g\in A$ and $m<\omega$ such that $f<^mg$ so that
        whenever $i\geq m$, $j\geq m$, and $\Phi_f(i,j)=\Phi_g(i,j)$, $\Psi(i,j)=\Phi_f(i,j)$.
        
    \end{proof}
    Now, for each $i$ with $s^\frown i\in q$, fix an $f_i$ such that for all $n<\omega$, $q^{[s^\frown i]}\not\Vdash\dot\Psi=^n\Phi_f$. 
    Let $f_s$ be a $<^*$ upper bound for $\{f_i\mid s^\frown i\in q\}$. 
    Then whenever $s^\frown i\in q$ and $n<\omega$, $q^{[s^\frown i]}\not\Vdash \dot\Psi=^n\Phi_{f_s}$. 
    Now, recursively define $q_i\leq q^{[s^\frown i]}$ and $n_i$ as in the claim, first forcing that $\dot\Psi\neq^{\max_{j<i}n_j}\Phi_f$ and then using that $\dot\Psi$ is forced to be a trivialization of $\Phi$ to decide some $n_i$ with $\dot\Psi=^{n_i}\Phi_f$. 
    Then $\bigcup_{i\in\operatorname{succ}_q(s)}q_i$ is a condition extending $q$ which still has $s$ as a splitting node and $f_s$ and $\langle n_i\rangle$ satisfy the Claim \ref{miller_pres_clm}. 
    \end{enumerate}
\end{proof}
\section{Questions}
While we have shown that derived limits do not necessarily vanish in many models with small continuum, the motivational question still remains of whether derived limits can vanish with a small continuum. 
One specific instance is the following:
\begin{question}
    Does $\mathfrak{d}=\aleph_2$ imply $\lim^2\mathbf{A}\neq0$?
\end{question}
The more general instance is
\begin{question}
    Does $\lim^n\mathbf{A}=0$ for all $n>0$ imply $\mathfrak{d}\geq\aleph_{\omega+1}$?
\end{question}
The analysis in Section \ref{Mitchell_section} of the first derived limit in the Mitchell model depended on a strong form of the fact that Cohen forcing does not trivialize any nontrivial $1$-coherent families. 
Whether the same is true for higher $n$ is unclear. 

\begin{question} \label{Cohen_no_2_quest}
    Can Cohen forcing add a trivialization to a nontrivial $2$-coherent family?
\end{question}
Much more generally than Question \ref{Cohen_no_2_quest}, we may ask whether any forcing can preserve nontrivial $2$-coherent families. 
Of course, forcings which add no new sets of reals cannot possibly add a trivialization, and by Theorem \ref{gob_vanish}, forcings which collapse $\mathfrak{d}$ to $\aleph_1$ trivialize every $2$-coherent family. 
\begin{question}
    Are there any forcings which add new sets of reals but do not trivialize any nontrivial $2$-coherent families?
\end{question}
Our next question is about the first derived limit in the Miller model. 
Obtaining models with a nonvanishing first derived limit with a large value of $\mathfrak{d}$ is a rather difficult task; the only model we know of where $\mathfrak{d}>\aleph_1$ but $\lim^1\mathbf{A}\neq0$ is the model of $MA$ produced by Todor\v{c}evi\'c in \cite{CFS}. 
\begin{question}
    Does $\lim^1\mathbf{A}=0$ in the Miller model?
\end{question}

\end{document}